\documentclass[reqno,12pt]{amsart}
\usepackage[utf8]{inputenc}
\usepackage[left=1.2in, right=1.2in, top=1in]{geometry}

\usepackage{amsmath, amssymb,amsthm, mathrsfs,color,times,textcomp,verbatim}
\usepackage{xcolor}
\usepackage[colorlinks=true]{hyperref}
\hypersetup{urlcolor=blue, citecolor=red, linkcolor=blue}
\usepackage[square,numbers]{natbib}
\usepackage{pgfplots}
\usepackage{tikz}
\numberwithin{equation}{section}
\theoremstyle{plain}
\newtheorem{theorem}{Theorem}[section]

\newtheorem{lemma}[theorem]{Lemma}

\newtheorem{claim}{Claim}

\usepackage{graphicx}

\theoremstyle{definition}

\newtheorem{remark}[theorem]{Remark}

\allowdisplaybreaks[4]
\title{Liouville type theorem for a class of quasilinear $p$-Laplace type equations in the half space}
\begin{document}
\author{Bao Yu }
\address{School of Mathematical Sciences, University of Science and Technology of China, Hefei, Anhui Province, P. R. China, 230006}
\email{baoyu1@mail.ustc.edu.cn}
\author{Yang Zhou}
\address{School of Mathematical Sciences, University of Science and Technology of China, Hefei, Anhui Province, P. R. China, 230006}
\email{zy19700816@mail.ustc.edu.cn}
\maketitle
\begin{abstract}
	We use the method of vector fields to obtain a Liouville-type theorem for a class of quasilinear $p$-Laplace type equations with Conormal boundary condition in the half space. These $p$-Laplace type equations are the subcritical case of the Euler-Lagrange equation of the Sobolev trace inequality in the half space.
\end{abstract}
\section{Introduction}
The research for Liouville-type theorem for the nonnegative solution to the following semilinear elliptic equations
\begin{equation}\label{1.1a}
	\Delta u+ u^q=0\,\,\,\,\,\textrm{in}\quad \mathbb{R}^{n}
\end{equation} 
in the range of $1 < q < 2^* - 1$ where $2^* = \frac{2n}{n-2}$ started in the splendid paper by Gidas and Spruck \cite{GS81}, and they found that these equations had no positive solution.   Gidas-Spruck \cite{GS81} proved their results via the method of vector fields and integral by parts motivated by Obata identity \cite{Obata71}. \\

The equation (\ref{1.1a}) had been studied intensively by many authors in
decades. In fact, for the critical case $q = 2^*-1$, it comes from the Yamabe problem on $\mathbb{R}^{n}$ and is also the Euler-Lagrange equation of Sobolev inequality. It's well known that the Aubin-Talenti bubbles (\ref{1}) are solutions of \eqref{1.1a} when $q = 2^*-1$ and are the minimizers of Sobolev inequality.
 \begin{equation}\label{1}
 	u(x)=\bigg(\frac{\lambda \sqrt{n(n-2)}}{\lambda^2+|x-x_0|^2}\bigg)^{\frac{n-2}{2}},~~ \lambda>0,~~ x_0\in \mathbb{R}^{n}.
 \end{equation}
Under the additional hypothesis $u(x)=O(|x|^{2-n})$, Obata \cite{Obata71} used vector fields method to show that the solution of \eqref{1.1a} must be of the form \eqref{1}. Later, Gidas, Ni and Nirenberg \cite{GNN79} developed the moving planes technique and also obtained this result under the same assumption. In 1989, Caffareli, Gidas and Spruck \cite{CGS89} (also seen in \cite{CL91}) used Kelvin
transform and the moving sphere method to classify the solution of \eqref{1.1a} without additional hypothesis. However, the moving sphere method relies on the inversion invariance of the Laplace equations, and fails in the $p$-Laplacian equations.\\

For the $p$-Laplacian equations
\begin{equation} \label{2}
	\Delta_p u+ u^q=0\,\,\,\,\,\textrm{in}\,\quad \mathbb{R}^{n},
\end{equation}
 Serrin and Zhou \cite{Serrin02} proved the similar Liouville theorem for the $p$-Laplacian in the subcritical case $p-1<q<p^*-1$ by using the method of vector fields motivated by Obata, where $p^* = \frac{np}{n-p}$. \\

 As for the critical case $q=p^*-1$, under the assumption of finite energy
 \begin{equation}
     u \in D^{1,p}(\mathbb{R}^n) :=\{u\in L^{p^*}(\mathbb{R}^n), Du \in L^p(\mathbb{R}^n) \}
 \end{equation}
 Vetois \cite{Vetois} and Sciunzi \cite{Sciunzi} used asymptotic analysis and moving plane method to prove that the positive solutions of the equation \eqref{2} must be of the following form
 \begin{equation}
     u(x) = \bigg(\frac{\lambda^{\frac{1}{p-1}}n^{\frac{1}{p}}(\frac{n-p}{p-1})^\frac{p-1}{p}}{|x-x_0|^\frac{p}{p-1}+\lambda^\frac{p}{p-1}}\bigg)^\frac{n-p}{p} ,~~ \lambda>0,~~ x_0\in \mathbb{R}^{n}.
 \end{equation}
 Note that the moving plane method relies
on the symmetries of the equation, so it can hardly work in the anisotropic setting.
In 2020, based on the
asymptotic estimates in \cite{Vetois,Sciunzi}, Ciraolo, Figalli and Roncoroni \cite{CFR20} used vector fields method to
classify positive solutions to the critical p-Laplacian equation
in an anisotropic setting.

These Liouville-type results have played a fundamental role in the study of semilinear elliptic equations with critical exponent, including the Yamabe problem, Nirenberg problem, and Sobolev inequality. \\

In this paper, we consider the following equation in the half space with Conormal boundary condition.
	\begin{equation}\label{pde1}
	\begin{cases}
		&\Delta_m u + u^p =0 , u>0,  ~~\text{in} ~~\mathbb{R}^{n}_+,\\
		&|\nabla u|^{m-2}\frac{\partial u}{\partial x_n} = -u^q, ~~\text{on} ~~ \partial \mathbb{R}^{n}_+,
	\end{cases}
\end{equation}
where $1< m<n,$ $0<p\leq \frac{nm}{n-m} -1$, $m-1<q\leq \frac{n(m-1)}{n-m}$. And $\mathbb{R}^{n}_+ =\{x=(x',x_n); x' \in \mathbb{R}^{n-1}, x_n>0\}$ denotes the Euclidean half space. 
In the critical case $p= \frac{nm}{n-m} -1, ~q= \frac{n(m-1)}{n-m}$, the equation \eqref{pde1} is the Euler-Lagrange equation associated with the following interpolation inequality between the Sobolev inequality and the Sobolev trace inequality. (see, e.g. \cite{BL95,Ne})
\begin{equation}
	\frac{||\nabla u||^m_{L^m(\mathbb{R}^{n}_+)}}{A}- \frac{|| u||^m_{L^{m^*}(\partial \mathbb{R}^{n}_+)}}{B} \geq ||u||^m_{L^{m^*}(\mathbb{R}^{n}_+)}
\end{equation}
for $m^*=\frac{nm}{n-m}$ and some appropriate constants $A>0, B>0$. Moreover, this equation is closely related to the second Yamabe problem. (see, e.g. \cite{Es92,BC14}) \\

When $m=2$ and $p,q$ are critical, that is $p=\frac{n+2}{n-2},q=\frac{n}{n-2}$,
Escobar \cite{E90} proved that any solution $u$ to the equation (\ref{pde1}) under the assumption $u(x)=O(|x|^{2-n})$ for
$n\ge 3$ and $|x |$ large, must be of the form 
\begin{equation}\label{e1.8}
	u(x',x_n)=\bigg(\frac{\mu}{1+\mu^2|x-x_0|^2}\bigg)^{\frac{n-2}{2}}, ~~~ \mu>0,~~ x_0\in \mathbb{R}^{n}_-.
\end{equation}
In 1995, Li and Zhu \cite{Li95} applied the moving sphere method to classify the positive solutions to the equation (\ref{pde1}) without additional assumption.
Later, Li and Zhang \cite{Li03} considered $m=2$ and $p,q$ subcritical, and used the moving plane method to show that the solution to (\ref{pde1}) was trivial.
For general $1<m<n$ , the second author \cite{Zhou24} established a classification of the solution for the critical case using the method of vector fields and integral by parts. \\

This paper aims to establish the Liouville theorem for the subcritical case of \eqref{pde1}. We now state our main result.

\begin{theorem}\label{Them1}
	If  $p,q$ satisfy that
	\begin{align}
		&\frac{n(m-1)}{n-m}\leq p<\frac{nm}{n-m}-1, \quad \frac{(n-1)(m-1)}{n-m}\leq q<\frac{(m-1)n}{n-m}, \quad q>\frac{n-2}{n-1}p-\frac{1}{n-1},\label{e1.9}\\
		&	\frac{(n-1)^2}{n^2}(p+1)^2 -2\bigg(\frac{n-1}{n}-(\frac{n-1}{n}-2\frac{m-1}{m})q \bigg)(p+1)+(1-q)^2 <0,
	\end{align}
	then there exists no nontrivial nonnegative $C^2$ solution of \eqref{pde1} in $\mathbb R_+^n$.
\end{theorem}

\begin{remark}
    For example, we consider $m=4,n=8$, then (\ref{e1.9}) becomes: $6\le p<7, 5.25\le q<6, q>\frac{6p-1}{7}$.
\begin{figure}[h!]
    \centering
    \includegraphics[width=0.5\linewidth]{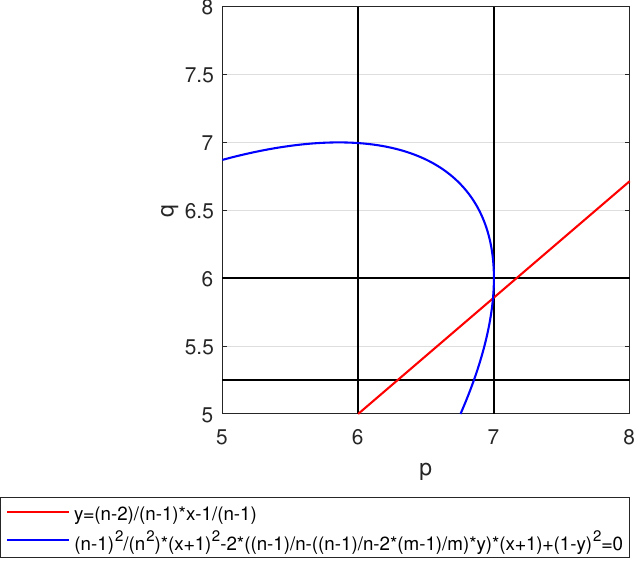}
    \caption{example: $m=4,n=8$}
    \label{fig1}
\end{figure}

\end{remark}

\begin{remark}
    It is not hard to prove the Liouville theorem in the Serrin index: $0<p < \frac{n(m-1)}{n-m}$, 
    $m-1<q< \frac{(n-1)(m-1)}{n-m}$. We handle the Serrin index in Section 2.
\end{remark}

The paper is organized as follows. In Section 2, we will introduce some notations and deduce the Liouville theorem for Serrin index. In Section 3, we will modify the integral identities used in \cite{Zhou24} to prove Theorem \ref{Them1}.

{\bf Acknowledgment:}
We would like to thank Prof. Xi-Nan Ma for his advanced guidance. This work was supported by National Natural Science Foundation of China [grant number 12141105].

\section{An Integral Equality }
In the whole paper, we denote by $B_r(x)$ the Euclidean ball centered at the point $x$ with radius $r$, and by $\Omega$ the upper half space $\mathbb{R}^n_+$.\\

First, we can deduce the following two lemmas, which show the Liouville theorem for a simple range of $p,q$.

\begin{lemma}\label{lem2.1}
	Let $u\geq 0$ satisfy $\Delta_m u + u^p =0$ in $\Omega$  and  $\frac{\partial u}{\partial x_n}\leq 0$ on $\partial \Omega$. If $0\le p < \frac{n(m-1)}{n-m}$, then $u=0$ in $\Omega$.
\end{lemma}
\begin{proof}
By Serrin's strong maximum principle \cite{S1970}, if $u(x_0)=0$ for some point $x_0\in \Omega$, then $u=0$ in $\Omega$. So we assume $u>0$ in $\Omega$.

	Let $\eta$ be a cut-off function such that $\eta = 1$ in $B_R(0)$, $\eta = 0$ in $B^c_{2R}(0)$ and $0\leq\eta \leq 1$, $|\nabla \eta|\leq \frac{C}{R}$. Multiplying \eqref{pde1} by $u^\alpha \eta^s$  and integrating by parts, where $\alpha,s \in \mathbb{R}$ are to be determined, we get
	\begin{align*}
		\int_\Omega u^{\alpha+p} \eta^s dx &= -\int_\Omega \Delta_m u u^\alpha \eta^s dx\\
		&=\int_\Omega|\nabla u|^{m-2}u_i (u^\alpha \eta^s)_i dx +
		\int_{\partial \Omega} |\nabla u|^{m-2}u_n u^\alpha \eta^s d\sigma\\
		&\leq \alpha \int_\Omega |\nabla u|^m u^{\alpha-1} \eta^sdx +
		\int_\Omega |\nabla u|^{m-1} u^\alpha (\eta^s)_i dx ,
	\end{align*}
	hence 
	\begin{equation}
		\int_\Omega u^{\alpha+p} \eta^s dx -\alpha \int_\Omega |\nabla u|^m u^{\alpha-1} \eta^sdx
		\leq\int_\Omega |\nabla u|^{m-1} u^\alpha (\eta^s)_i dx,
	\end{equation}
	where $\alpha <0$ is to be determined. Then we can use the following Young's inequality
	\begin{align}
		s|\nabla u|^{m-1} u^\alpha \eta^{s-1} |\nabla \eta| &\leq \frac{-\alpha}{2} |\nabla u|^mu^{\alpha-1} \eta^s + C u^{\alpha-1 +m} \eta^{s-m} |\nabla \eta|^m
	\end{align}
    to get 
    \begin{align}
        \int_\Omega u^{\alpha+p} \eta^s dx 
		\leq C\int_\Omega u^{\alpha-1 +m} \eta^{s-m} |\nabla \eta|^m dx.\label{e2.3}
    \end{align}
	\textbf{Case 1:} If $p>m-1$, we take $\alpha<0 $ big enough, such that $0< \alpha +m -1 < \alpha +p $ ;\\
	\textbf{Case 2:} If $p=m-1$, we take $\alpha= 1-m$, such that $ \alpha +m -1 =\alpha + p =0$ ;\\
	\textbf{Case 3:} If $p<m-1$, we take $\alpha<0 $ small enough, such that $\alpha + p < \alpha+m -1 <0 $ ;\\
In the case 2, the inequality (\ref{e2.3}) becomes
\begin{align}
    R^n\le C R^{n-m},
\end{align}
   which is impossible as $R\rightarrow\infty$.\\
	In the cases 1 and 3, we can apply the following Young's inequality in (\ref{e2.3})
	\begin{equation}
		C u^{\alpha+m-1}\eta^{s-m}|\nabla \eta|^m \leq \frac{1}{2} u^{\alpha+p} \eta^s
		+C|\nabla \eta |^{m\frac{\alpha+p}{p+1-m}},
	\end{equation}
    and obtain
    \begin{align}
         \int_\Omega u^{\alpha+p} \eta^s dx 
		\leq C R^{n-m\frac{\alpha+p}{p+1-m}}.
    \end{align}
	So it is sufficient to show
	\begin{equation}\label{res}
		-m\frac{\alpha+p}{p+1-m} +n <0.
	\end{equation}
	In the case when $p>m-1$, (\ref{res}) is equivalent to $p<\frac{n(m-1)+m\alpha}{n-m}$. Taking the limit as $\alpha \to 0_-$, this yields 
	$p<\frac{n(m-1)}{n-m}$. In the case when $p< m-1$, we can let $\alpha \to - \infty$, which ensures that the inequality \eqref{res} 
	always holds. Consequently, we establish the Liouville theorem for $0\le p< \frac{n(m-1)}{n-m}$.
\end{proof}

\begin{lemma}
	Let $u\geq 0$ satisfy $\Delta_m u \leq 0$ in $\Omega$   and   $|\nabla u|^{m-2}\frac{\partial u}{\partial x_n}= -u^q$ on $\partial \Omega$. If $m-1\leq q < \frac{(n-1)(m-1)}{n-m}$, then $u=0$.
\end{lemma}
\begin{proof} As in lemma \ref{lem2.1}, we assume $u > 0$ in $\Omega$.
	Let $\eta$ be a cut-off function such that $\eta = 1$ in $B_R(0)$, $\eta = 0$ in $B^c_{2R}(0)$ and $0\leq\eta \leq 1$, $|\nabla \eta|\leq \frac{C}{R}$. Multiplying \eqref{pde1} by $u^\alpha \eta^s$  and integrating by parts, where $\alpha,s \in \mathbb{R}$ are to be determined, we get
	\begin{align}
		0&\geq \int_\Omega \Delta_mu u^\alpha \eta^s \notag \\
		&=\int_\Omega -\alpha |\nabla u|^m u^{\alpha-1}\eta^s -\int_\Omega s|\nabla u|^{m-2}
		u^\alpha \eta^{s-1} u_i \eta_i -\int_{\partial \Omega} |\nabla u|^{m-2}\frac{\partial u}{\partial x_n}u^\alpha \eta^s.
	\end{align}
    Hence we get
    \begin{equation} \label{eq8}
    	\int_{\partial \Omega} u^{\alpha+q } \eta^s -\alpha \int_\Omega|\nabla u|^mu^{\alpha-1}\eta^s \leq s\int_\Omega |\nabla u|^{m-1} u^\alpha \eta^{s-1}|\nabla \eta |.
    \end{equation}
   We can use the following Young's inequality
   \begin{equation}
   	s\int_\Omega |\nabla u|^{m-1} u^\alpha \eta^{s-1}|\nabla \eta | \leq -\frac{\alpha}{2}
   	\int_\Omega|\nabla u|^mu^{\alpha-1}\eta^s + C \int_\Omega u^{\alpha-1+m} \eta^{s-m}|\nabla \eta |^m.
   \end{equation}
   Substituting into \eqref{eq8}, we deduce
   \begin{equation}
   	\int_{\partial \Omega} u^{\alpha+q } \eta^s \leq C\int_\Omega u^{\alpha-1+m} \eta^{s-m}|\nabla \eta |^m.
   \end{equation}
Note that using Serrin's strong maximum principle \cite{S1970}, if $u\geq 0$ satisfy $\Delta_m u \leq 0$ in $\Omega$   and   $|\nabla u|^{m-2}\frac{\partial u}{\partial x_n}\le 0$ on $\partial \Omega$, we have
   \begin{equation}
       u(x) \geq C |x|^{-\frac{n-m}{m-1}} \quad\text{for}\quad|x|\,\,\,\text{large}.\label{e2.12}
   \end{equation}
Then choosing $\alpha =1-m, s=m$, and using (\ref{e2.12}), we obtain 
   \begin{equation}
   	R^{n-1}R^{-\frac{n-m}{m-1}(1-m+q)}\leq C R^{n-m}.
   \end{equation}
   Letting $R\to \infty$, we get $q\geq \frac{(m-1)(n-1)}{n-m}$. Hence if $m-1\le q< \frac{(m-1)(n-1)}{n-m}$, then we have
   $u=0$ in $\Omega$.
\end{proof}

\section{Proof of the theorem}

Denote $p^*= \frac{nm}{n-m}-1, q^*=\frac{n(m-1)}{n-m}$, which is called Sobolev index. Denote $p_*=\frac{n(m-1)}{n-m}$, $q_* =\frac{(n-1)(m-1)}{n-m}$, which is called Serrin index. The work we are going to do next is to deal with the region $p_* \leq p < p^*, q_* \leq q < q^*$. Now we will establish two differential identities adapted from \cite{Zhou24}.

\begin{proof}[Proof of Theorem \ref{Them1}]
    Set $X^i=|\nabla u|^{m-2}u_i$, $X^i_j=(|\nabla u|^{m-2}u_i)_j$, $E^i_j= X^i_j -\frac{\Delta_m u}{n}\delta_{ij} $, $L_{ij}=\frac{|\nabla u|^{m-2}u_iu_j}{u} -\frac{1}{n}
\frac{|\nabla u|^m}{u}\delta_{ij}$. Since $E^i_j$ and $L_{ij}$ are both trace-free, it is easy to check
\begin{align*}
	E^i_jE^j_i&=X^i_j X^j_i-\frac{(\Delta_m u)^2}{n} ,\\
    E^i_jL_{ij}&=\frac{|\nabla u|^{m-2}X^i_ju_iu_j}{u}-\frac{\Delta_m u |\nabla u|^m}{nu} ,\\
    L_{ij}^2&=\frac{|\nabla u|^{2m}}{u^2}-\frac{1}{n}\frac{|\nabla u|^{2m}}{u^2}=\frac{n-1}{n}\frac{|\nabla u|^{2m}}{u^2} .
\end{align*}

	Set $\alpha=1-q $, we consider the following Serrin-Zou type differen-
tial equality
\begin{align}
	&\partial_i(u^\alpha X^i_j X^j - u^\alpha \Delta_m u X^i +\mu u^{\alpha-1}|\nabla u|^mX^i -\lambda u^{\alpha+p}X^i)\notag\\
	=&u^\alpha(X^i_jX^j_i-(\Delta_m u)^2)+ \alpha u^{\alpha-1}X^i_jX^ju_i-\alpha u^{\alpha-1}\Delta_m u|\nabla u|^m + \mu u^{\alpha-1}|\nabla u|^m \Delta_m u \notag\\
	&+m\mu u^{\alpha-1}|\nabla u|^{m-2} u_ku_{ki}X^i +\mu (\alpha-1)u^{\alpha-2}|\nabla u|^{2m} \notag \\
	&-\lambda(\alpha+p)u^{\alpha+p-1}|\nabla u|^{m} -\lambda u^{\alpha+p}\Delta_m u\notag\\
	=&u^\alpha(E^i_jE^j_i+\frac{(\Delta_m u)^2}{n})- u^\alpha (\Delta_m u)^2 +(\frac{m}{m-1}\mu +\alpha)
	u^\alpha(E^i_jL_{ij}+\frac{\Delta_m u |\nabla u|^m}{nu})\notag\\
	&+(\mu-\alpha) u^{\alpha-1}\Delta_m u|\nabla u|^m
	+\mu (\alpha-1)u^{\alpha-2}|\nabla u|^{2m}
	-\lambda(\alpha+p)u^{\alpha+p-1}|\nabla u|^m\notag\\
	& -\lambda u^{\alpha+p}\Delta_m u\notag\\
	=&u^\alpha|E^i_j+\frac{1}{2}(\frac{m}{m-1}\mu+\alpha)L_{ij}|^2 +(\lambda-\frac{n-1}{n})u^\alpha(\Delta_m u)^2\notag \\ &+\bigg((\mu-\alpha)+\frac{1}{n}(\frac{m}{m-1}\mu+\alpha)+\lambda(\alpha+p)\bigg)
	u^{\alpha-1}\Delta_m u |\nabla u |^m\notag\\
	&+\bigg(\mu(\alpha-1)-\frac{n-1}{n}\frac{1}{4}(\frac{m}{m-1}\mu+\alpha)^2\bigg)u^{\alpha-2}|\nabla u|^{2m}. \label{eq2}
\end{align}
Multiplying both sides by $\eta$ and integrating by parts, where $\eta$ is a cut-off function to be chosen later, we obtain
\begin{align}\label{eq3}
	&\int_\Omega\partial_i(u^\alpha X^i_jX^j - u^\alpha \Delta_m u X^i +\mu u^{\alpha-1}|\nabla u|^m X^i -\lambda u^{\alpha+p}X^i)\eta dx\notag \\
	=&\int_\Omega -u^\alpha X^i_jX^j\eta _i + u^\alpha \Delta_m u X^i\eta _i -\mu u^{\alpha-1}|\nabla u|^mX^i\eta _i +\lambda u^{\alpha+p}X^i\eta _i dx\notag\\
	&+\int_{\partial \Omega}-u^\alpha X^n_jX^j\eta + u^\alpha \Delta_m u X^n \eta -\mu u^{\alpha-1}|\nabla u|^mX^n\eta +\lambda u^{\alpha+p}X^n\eta.
\end{align}
Using the boundary condition $X^n=-u^q$ on $\partial\Omega$, we rewrite the last term in $(\ref{eq3})$ as follows.
\begin{align}
	&\int_{\partial \Omega}-u^\alpha X^n_jX^j\eta + u^\alpha \Delta_m u X^n \eta -\mu u^{\alpha-1}|\nabla u|^mX^n\eta +\lambda u^{\alpha+p}X^n\eta\notag\\
	=& \int_{\partial \Omega} \sum_{a=1}^{n-1}-u^\alpha (|\nabla u|^{m-2}u_n)_a|\nabla u|^{m-2}u_a \eta -\int_{\partial \Omega} u^\alpha (|\nabla u|^{m-2}u_n)_n|\nabla u|^{m-2}u_n \eta\notag \\
	&+u^\alpha \bigg[\sum_{a=1}^{n-1}
	(|\nabla u|^{m-2}u_a)_a+(|\nabla u|^{m-2}u_n)_n\bigg]|\nabla u|^{m-2}u_n \eta\notag\\
	&-\mu u^{\alpha-1}|\nabla u|^m|\nabla u|^{m-2}u_n\eta +\lambda u^{\alpha+p}|\nabla u|^{m-2}u_n\eta\notag\\
	=&\int_{\partial \Omega} u^\alpha (u^q)_a|\nabla u|^{m-2}u_a \eta -u^\alpha (|\nabla u|^{m-2}u_a)_a u^q \eta +\mu u^{\alpha-1}|\nabla u|^m u^q \eta
	-\lambda u^{\alpha+p}u^q \eta\notag\\
	=&\int_{\partial \Omega} q|\nabla u|^{m-2}(u_a)^2\eta +|\nabla u|^{m-2} (u_a)^2\eta +u|\nabla u|^{m-2}u_a \eta _a + \mu |\nabla u|^m \eta - \lambda
	u^{p+1}\eta \notag\\
	=&\int_{\partial \Omega} (q+1)|\nabla u|^{m-2}(u_a)^2\eta  +u|\nabla u|^{m-2}u_a \eta _a + \mu |\nabla u|^m \eta - \lambda
	u^{p+1}\eta.\label{e3.3}
\end{align}
To handle these boundary terms, we need the following Pohozaev-type identity.
\begin{align*}
	&\partial_n(|\nabla u|^m)-m\partial_i(|\nabla u|^{m-2}u_iu_n)-\frac{m}{p +1}\partial_n(u^{p+1})\\
	=&m|\nabla u|^{m-2}u_iu_{in}-m(|\nabla u|^{m-2}u_i)_iu_n-m|\nabla u|^{m-2}u_iu_{ni}-mu^p u_n\\
	=&0.
\end{align*}
Multiplying both sides by $\eta$ and integrating by parts, we get
\begin{align}\label{eq1}
	0&=\int_\Omega \bigg(\partial_n(|\nabla u|^m)-m\partial_i(|\nabla u|^{m-2}u_iu_n)-\frac{m}{p +1}\partial_n(u^{p+1})\bigg) \eta dx\notag \\
	&=\int_\Omega -|\nabla u|^m\eta _n + m|\nabla u|^{m-2}u_i  u_n\eta _i +\frac{m}{p+1}u^{p+1}\eta _n\notag \\
	&+\int_{\partial \Omega} -|\nabla u|^m \eta + m|\nabla u|^{m-2}u_n^2 \eta +\frac{m}{p+1}u^{p+1}\eta d\sigma.
\end{align}
\textbf{Case 1:} $\frac{n-1}{n}<\frac{q+1}{p+1}$.\\
Taking $\mu =-\frac{m-1}{m}(q+1)$ and using \eqref{eq1}, we have
\begin{align*}
	&\int_{\partial \Omega} -(q+1)|\nabla u|^{m-2}(u_a)^2\eta   - \mu |\nabla u|^m \eta\\
	=&-(\mu+q+1)\int_{\partial \Omega} |\nabla u|^m\eta +(q+1)\int_{\partial \Omega} |\nabla u|^{m-2}u_n^2\eta\\
	=&\frac{q+1}{m}\bigg[-\int_{\partial \Omega} |\nabla u|^m\eta +m\int_{\partial \Omega} |\nabla u|^{m-2}u_n^2\eta  \bigg]\\
	=&\frac{q+1}{m}\bigg[\int_\Omega |\nabla u|^m\eta _n - m|\nabla u|^{m-2}u_i  u_n\eta _i -\frac{m}{p+1}u^{p+1}\eta _n-\int_{\partial \Omega}\frac{m}{p+1}u^{p+1}\eta d\sigma\bigg].
\end{align*}
Substituting into \eqref{eq3}, we get
\begin{align}\label{eq4}
	&\int_\Omega \bigg\{u^\alpha|E^i_j+\frac{1}{2}(\frac{m}{m-1}\mu+\alpha)L_{ij}|^2 \eta +(\lambda-\frac{n-1}{n})u^\alpha(\Delta_m u)^2\eta\notag \\ &+\bigg((\mu-\alpha)+\frac{1}{n}(\frac{m}{m-1}\mu+\alpha)+\lambda(\alpha+p)\bigg)
	u^{\alpha-1}\Delta_m u |\nabla u |^m \eta \notag\\
	&+\bigg(\mu(\alpha-1)-\frac{n-1}{4n}(\frac{m}{m-1}\mu+\alpha)^2\bigg)u^{\alpha-2}|\nabla u|^{2m} \eta\bigg\}+\int_{\partial \Omega}(\lambda- \frac{q+1}{p+1}) u^{p+1}\eta \notag\\
	=&\int_\Omega -u^\alpha X^i_jX^j\eta _i + u^\alpha \Delta_m u X^i\eta _i -\mu u^{\alpha-1}|\nabla u|^mX^i\eta _i +\lambda u^{\alpha+p}X^i\eta _i dx\notag\\
	&-\bigg[\frac{q+1}{m}\int_\Omega |\nabla u|^m \eta _n -(q+1)\int_\Omega |\nabla u|^{m-2}
	u_iu_n\eta _i -\frac{q+1}{p+1}\int_\Omega u^{p+1} \eta _n \bigg]\notag\\
    &+\int_{\partial \Omega}u\sum_{a=1}^{n-1}|\nabla u|^{m-2}u_a\eta_a
\end{align}
In this identity, we hope the left side is positive and the right side converges to zero when $R\rightarrow\infty$.

Therefore, we need 
\begin{align}
	&\lambda >\frac{q+1}{p+1},\\
	&(\mu-\alpha)+\frac{1}{n}(\frac{m}{m-1}\mu+\alpha)+\lambda(\alpha+p)<0 ,\label{eq9}\\
	&\mu(\alpha-1)-\frac{n-1}{4n}(\frac{m}{m-1}\mu+\alpha)^2 >0. \label{eq10}
\end{align}
If we take $\lambda = \frac{q+1}{p+1}$, then \eqref{eq9} becomes 
\begin{equation}\label{eq11}
	q^2-(\frac{1}{m}-\frac{2}{n})q-(1+\frac{1}{m}-\frac{2}{n})pq+\frac{m-1}{m}(p+1) >0.
\end{equation}
\begin{claim}
	If $\lambda = \frac{q+1}{p+1}+\delta$ and $\delta>0$ small enough depending only on $n,m,p,q$, then \eqref{eq10} and \eqref{eq11} are satisfied under the assumptions:
	 $p_* \leq p \leq p^*, q_* \leq q \leq q^*$ and $\frac{n-1}{n}<\frac{q+1}{p+1}$.
\end{claim}
\begin{proof} By continuity, we only consider $\delta=0$. \\
	Firstly, we check \eqref{eq10}. Since $\mu= -\frac{m-1}{m}(q+1)$, (\ref{eq11}) becomes
	\begin{align*}
		&\frac{m-1}{m}(q+1)q-\frac{n-1}{4n}(-q-1+1-q)^2 >0\\
		\Longleftrightarrow & \frac{m-n}{nm}q^2+\frac{m-1}{m}q >0\\
		\Longleftrightarrow &  0<q<\frac{n(m-1)}{n-m}.
	\end{align*}
   Hence \eqref{eq10} is valid.\\ 
   Next, we check \eqref{eq11}.
   Note that \eqref{eq11} represents a hyperbola passing through the point $(p^*,q^*)$.  And its tangent line at this point is $(q-q^*)= \frac{(n-1)(m-1)}{mn-2n+m}(p-p^*)$. Comparing it with the line
$\frac{n-1}{n}=\frac{q+1}{p+1}$ which also passes the point $(p^*,q^*)$, we can see
   \begin{equation}
   	\frac{(n-1)(m-1)}{mn-2n+m}>\frac{n-1}{n}\Longleftrightarrow m<n .
   \end{equation}
   So \eqref{eq11} is valid under the assumption $\frac{n-1}{n}<\frac{q+1}{p+1}$.
\end{proof}

\noindent\textbf{Case 2:}$\frac{n-1}{n}\geq \frac{q+1}{p+1}$\\
Note that the boundary term $\int_{\partial \Omega} u^{p+1}\eta d\sigma$ can also be written as follows:
\begin{align} \label{20}
	\int_{\partial \Omega} u^{p+1}\eta d\sigma
	&=\int_\Omega \frac{p+1}{m}\bigg(|\nabla u|^m \eta _n - m|\nabla u|^{m-2}u_i u_n\eta _i-\frac{m}{p+1}u^{p+1}\eta _n \bigg) \notag\\
	&+\int_{\partial \Omega}\frac{p+1}{m} \bigg( |\nabla u|^m\eta -m|\nabla u|^{m-2}u_n^2 \eta \bigg).
\end{align}
Therefore, the right side in (\ref{e3.3}) becomes
\begin{align}
	&-\lambda\int_\Omega \frac{p+1}{m}\bigg(|\nabla u|^m \eta _n - m|\nabla u|^{m-2}u_i u_n\eta _i-\frac{m}{p+1}u^{p+1}\eta _n \bigg) \notag\\
    &-\bigg[\frac{p+1}{m}\lambda \int_{\partial \Omega} |\nabla u|^m\eta d\sigma
	-(p+1)\lambda\int_{\partial \Omega} |\nabla u|^{m-2}u_n^2\eta d\sigma-\mu\int_{\partial \Omega} |\nabla u|^m\eta d\sigma\notag \\
	&-(q+1)\int_{\partial \Omega} \sum_{a=1}^{n-1}|\nabla u|^{m-2}u_a^2\eta d\sigma 
	 \bigg]_A +\int_{\partial \Omega} u|\nabla u|^{m-2}u_a \eta _a\label{eq7}
\end{align}
In this case, let $\lambda >\frac{n-1}{n}\geq  \frac{q+1}{p+1}$, we have $(p+1)\lambda > q+1$. Thus, the boundary term A in \eqref{eq7} is 
\begin{equation}
	-\bigg[ \bigg(\frac{p+1}{m}\lambda -\mu - (p+1)\lambda\bigg)\int_{\partial \Omega}
	|\nabla u|^m \eta +\bigg((p+1)\lambda -(q+1) \bigg)\int_{\partial \Omega}
	|\nabla u|^{m-2}u_a^2\eta d\sigma \bigg].
\end{equation}
Hence we need
\begin{equation} \label{line1}
	\mu +\frac{m-1}{m}(p+1)\lambda \leq 0.
\end{equation}
To make the left side of \eqref{eq4} positive, we need
\begin{align}
	&\lambda >\frac{n-1}{n},\\
	&(\mu-\alpha)+\frac{1}{n}(\frac{m}{m-1}\mu+\alpha)+\lambda(\alpha+p)<0,\label{line2}\\
	&\mu(\alpha-1)-\frac{n-1}{4n}(\frac{m}{m-1}\mu+\alpha)^2 >0. \label{line3}
\end{align}
\begin{claim}
	If we choose $\lambda= \frac{n-1}{n}+\delta, \mu=-\frac{m-1}{m}(p+1)\frac{n-1}{n}- \delta$, where
	$\delta$ are small enough depending only on $n,m,p,q$, then \eqref{line1} and \eqref{line2} are satisfied. Under the addition assumption
	\begin{equation} \label{eq12}
		\frac{(n-1)^2}{n^2}(p+1)^2 -2\bigg(\frac{n-1}{n}-(\frac{n-1}{n}-2\frac{m-1}{m})q \bigg)(p+1)+(1-q)^2 <0,
	\end{equation}
\eqref{line3} is satisfied.
\end{claim}
\begin{proof}
	 \eqref{line1} is obvious. By continuity, we only consider $\delta=0$.
     
     We now check \eqref{line2}. Since $ \lambda= \frac{n-1}{n}$ and $\mu=-\frac{m-1}{m}(p+1)\frac{n-1}{n}$, \eqref{line2} becomes
	 \begin{align*}
	 	&-(1+\frac{m}{n(m-1)})\frac{m-1}{m}(p+1)\frac{n-1}{n} +\frac{n-1}{n}p <0 \\
	 	\Longleftrightarrow 
	 	&(1-\frac{nm-n+m}{nm})p <\frac{nm-n+m}{nm}\\
	 	\Longleftrightarrow
	 	&p<\frac{nm-n+m}{n-m}.
	 \end{align*}
 Next we check \eqref{line3}. We take $ \lambda= \frac{n-1}{n}$ and $\mu=-\frac{m-1}{m}(p+1)\frac{n-1}{n}$ and  \eqref{line3} follows that
 \begin{align*}
 	&\frac{m-1}{m}(p+1)\frac{n-1}{n}q-\frac{n-1}{4n} \bigg(-(p+1)\frac{n-1}{n}+1-q \bigg)^2 >0\\
 	\Longleftrightarrow
 	&\frac{(n-1)^2}{n^2}(p+1)^2 -2\bigg(\frac{n-1}{n}-(\frac{n-1}{n}-2\frac{m-1}{m})q \bigg)(p+1)+(1-q)^2 <0.
 \end{align*}\\
\end{proof}

Combining the two cases, we choose $\lambda= \frac{n-1}{n}+\delta, \mu=-\frac{m-1}{m}(p+1)\frac{n-1}{n}- \delta$, where
	$\delta$ are small enough depending only on $n,m,p,q$. Then from \eqref{eq4},
under the assumption of $p,q$ in Theorem \ref{Them1}, there exists $\epsilon>0$ depending only on $n,m,p,q,\delta$ such that
\begin{align}
	&\epsilon\bigg\{\int_\Omega \bigg( u^\alpha E^i_j E^j_i+ u^{\alpha-2}|\nabla u|^{2m}+ u^{\alpha+2p}\bigg)\eta + \int_{\partial \Omega}\big( u^{p+1}+|\nabla u|^m\big)\eta\bigg\} \notag\\
	\leq&C\int_\Omega \bigg(u^\alpha |X^i_jX^j\eta _i| + u^\alpha |\Delta_m u X^i\eta _i| + u^{\alpha-1}|\nabla u|^m|X^i\eta _i| +u^{\alpha+p}|X^i\eta _i|\bigg) dx\notag\\
	&+C\int_{\partial \Omega}u\sum_{a=1}^{n-1}|\nabla u|^{m-2}|u_a\eta _a| +C\int_\Omega\bigg[ |\nabla u|^m |\eta_n| + |\nabla u|^{m-2}
	|u_iu_n\eta _i| + u^{p+1} |\eta _n |\bigg] \notag \\
	\leq &C\int_\Omega \bigg(u^\alpha |E^i_j X^j\eta_i | + u^{\alpha-1}|\nabla u|^{2m-1}|\nabla\eta| +u^{\alpha+p}|X^i\eta _i|\bigg) dx\notag\\
	&+C\int_{\partial \Omega}u\sum_{a=1}^{n-1}|\nabla u|^{m-2}|u_a\eta _a| +C\int_\Omega\bigg[ |\nabla u|^m |\nabla\eta| + u^{p+1} |\nabla\eta |\bigg],
\end{align}
where $C$ depends only on $n,m,p,q$.\\
 We deal with the right-hand side using Young's inequality. Let $\eta = \phi^s$, and $\phi$ be a cut-off function such that $\phi = 1$ in $B_R(0)$, $\phi = 0$ in $B^c_{2R}(0)$ and $0\leq\phi \leq 1$, $|\nabla \phi|\leq \frac{C}{R}$.
 \begin{equation}
 	C\int_\Omega u^\alpha |E^i_j X^j \eta_i| \leq \frac{\epsilon}{4}\int_\Omega  u^\alpha E^i_j E^j_i \eta
 	+\frac{\epsilon}{4}\int_\Omega u^{\alpha-2} |\nabla u|^{2m}\eta +C \int_\Omega u^{\alpha-2+2m} \phi^{s-2m}|\nabla \phi|^{2m},\notag
 \end{equation}
\begin{equation}
    C\int_\Omega u^{\alpha+p}|X^i\eta_i| \leq \frac{\epsilon}{4}\int_\Omega  u^{\alpha+2p}\eta 
    +\frac{\epsilon}{4}\int_\Omega u^{\alpha-2} |\nabla u|^{2m}\eta +C \int_\Omega u^{\alpha-2+2m} \phi^{s-2m}|\nabla \phi|^{2m},\notag
\end{equation}
\begin{equation}
	C\int_\Omega u^{\alpha-1}|\nabla u|^{2m-1} |\nabla\eta| \leq \frac{\epsilon}{4}\int_\Omega u^{\alpha-2} |\nabla u|^{2m}\eta +C \int_\Omega u^{\alpha-2+2m} \phi^{s-2m}|\nabla \phi|^{2m},\notag
\end{equation}
\begin{equation}
	C\int_\Omega |\nabla u|^m |\nabla\eta| \leq \frac{\epsilon}{4}\int_\Omega u^{\alpha-2} |\nabla u|^{2m}\eta +C \int_\Omega u^{2-\alpha} \phi^{s-2}|\nabla \phi|^{2},\notag
\end{equation}
\begin{equation}
	C\int_\Omega u^{p+1}|\nabla\eta| \leq \frac{\epsilon}{4}\int_\Omega  u^{\alpha+2p} \eta
	+C \int_\Omega u^{2-\alpha} \eta^{s-2}|\nabla \phi|^{2},\notag
\end{equation}
\begin{equation}
	C\int_{\partial \Omega} u\sum_{a=1}^{n-1}|\nabla u|^{m-2}|u_a\eta_a| 
	\leq \frac{\epsilon}{2}\int_{\partial \Omega} |\nabla u|^m\eta +C\int_{\partial \Omega}
	u^m \phi^{s-m}|\nabla \phi|^m.\notag
\end{equation}
Consequently, we obtain
\begin{align}
	\int_\Omega u^{\alpha+2p}\phi^s +\int_{\partial \Omega} u^{p+1}\phi^s 
	\leq&  C \int_\Omega (u^{\alpha-2+2m} \phi^{s-2m}|\nabla \phi|^{2m} + u^{2-\alpha} \phi^{s-2}|\nabla \phi|^{2})\notag\\
    &+C\int_{\partial \Omega}u^m \phi^{s-m} |\nabla \phi|^m.\label{e3.20}
\end{align}
Here $s>2m$ is determined later.
Next, we will use the Young inequality to deal with the right-hand side of (\ref{e3.20}).\\

For the second term in (\ref{e3.20})
\begin{align*}
	C\int_\Omega u^{2-\alpha} \phi^{s-2}|\nabla \phi |^2 &= C\int_\Omega u^{q+1} \phi^{s-2}   |\nabla \phi |^2\\
	&\leq \frac{1}{4} \int_\Omega u^{2p-q+1} \phi ^s  + C \int_\Omega \phi^{s-\frac{2p-q+1}{p-q}} |\nabla \phi |^{\frac{2p-q+1}{p-q}}\\
	&\leq  \frac{1}{4} \int_\Omega u^{2p-q+1}\phi^s + C R^{n-\frac{2p-q+1}{p-q}}.
\end{align*}
It is sufficient to show that 
\begin{align}
	&n-\frac{2p-q+1}{p-q}<0 \notag\\
	\Longleftrightarrow &  ~q>\frac{n-2}{n-1}p-\frac{1}{n-1}.
\end{align}

For the third term in (\ref{e3.20})
\begin{align*}
	C\int_{\partial \Omega} u^m \phi^{s-m}|\nabla \phi|^m 
	 &\leq \frac{1}{2}
	\int_{\partial \Omega} u^{p+1}\phi^s +C\int_{\partial \Omega} \phi^{s-\frac{m(p+1)}{p+1-m}} |\nabla \phi|^{\frac{m(p+1)}{p+1-m}}\\
	&\leq \frac{1}{2}\int_{\partial \Omega} u^{p+1}\phi^s +C R^{n-1-\frac{m(p+1)}{p+1-m}}.
\end{align*}
It is sufficient to show that 
\begin{align*}
	&n-1-m\frac{p+1}{p+1-m}<0\\
	\Longleftrightarrow & (n-m-1)(p+1)<(n-1)m.
\end{align*}
If $n-m-1 \leq 0$, the above inequality is true since the left-hand side is negative. If $n-m-1 >0$,
then it is also true since  $p<\frac{nm-n+m}{n-m}<\frac{nm-n+1}{n-m-1}$.\\

For the first term in (\ref{e3.20}), we need to take care of the sign of $\alpha-2+2m = 2m-q-1$.
If $2m-q-1 >0$, then we use the Young inequality
\begin{align*}
	C\int_\Omega u^{\alpha-2+2m} \phi^{s-2m}|\nabla \phi |^{2m} &=  C\int_\Omega u^{2m-1-q} \phi^{s-2m}|\nabla \phi |^{2m}\\
	&\leq 
	\frac{1}{4}\int_\Omega u^{2p-q+1}\phi^s  +C\int_\Omega \phi^{s-m\frac{2p-q+1}{p-m+1}} |\nabla \phi|^{m\frac{2p-q+1}{p-m+1}}\\
	&\leq \frac{1}{4}\int_\Omega u^{2p-q+1}\phi^s +C R^{n-m\frac{2p-q+1}{p-m+1}}.
\end{align*}
It is sufficient to show that 
\begin{align}
	&n-m\frac{2p-q+1}{p-m+1}<0 \notag\\
	\Longleftrightarrow &q<\frac{2m-n}{m}p+\frac{mn+m-n}{m}. \label{eq13}
\end{align}
By direct calculation, we have
\begin{claim}
    In the region $\frac{n(m-1)}{n-m} \leq p <\frac{nm}{n-m}-1$, $ \frac{(n-1)(m-1)}{n-m}\leq q<\frac{n(m-1)}{n-m}$, the inequality  \eqref{eq13} is valid.
\end{claim}
%\begin{proof}
   % If we consider the boundary of \eqref{eq13} as a line, it is sufficient to show that on this line, when $p=\frac{n(m-1)}{n-m} $ and $\frac{nm}{n-m}-1 $, $q>\frac{n(m-1)}{n-m}$. By direct calculation, it follows that if $p=\frac{n(m-1)}{n-m}$, then $ q=\frac{m(n-1)}{n-m}$ and if $p=\frac{nm}{n-m}-1 $, then $ q=\frac{nm}{n-m}-1$.
%\end{proof}

If $2m-q-1\leq 0$, then we use the property $u(x) \geq C|x|^{-\frac{n-m}{m-1}}$ for $|x|$ large enough. One can arrive that
\begin{equation}
    \int_{\Omega} u^{\alpha-2+2m}\phi^{s-2m}|\nabla \phi|^{2m} \leq C R^{n-2m-(2m-1-q)\frac{n-m}{m-1}}.
\end{equation}
It is sufficient to show that
\begin{align}
    &n-2m-(2m-1-q)\frac{n-m}{m-1} <0 \notag \\
    \Longleftrightarrow & q<\frac{m(n-1)}{n-m}.
\end{align}
It is valid since we consider the region $q<\frac{n(m-1)}{n-m}$.

As $R \to \infty$, we can obtain $\int_\Omega u^{2p-q+1} =0$. Finally, we have proved 
the theorem.

\end{proof}

\end{document}